\documentclass[11pt]{amsart}

\usepackage{amscd,amssymb,amsopn,amsmath,amsthm,mathrsfs,graphics,amsfonts,enumerate,verbatim,calc
}
\usepackage{bbm}
\usepackage[all,cmtip]{xy}
\usepackage[all]{xy}
\usepackage{tikz}
\usepackage{lscape}
\usepackage{enumitem}
\usetikzlibrary{matrix,arrows,decorations.pathmorphing,cd}
\usepackage{hyperref}
\hypersetup{colorlinks=true,linkcolor={blue}}
\usepackage{cleveref}
\usepackage{url}

\usepackage{scalerel}
\usepackage{stackengine,wasysym}
\usetikzlibrary {positioning}
\usetikzlibrary{patterns}
\usetikzlibrary{calc}
\definecolor {processblue}{cmyk}{0.96,0,0,0}

\usepackage{subfigure}

\usepackage{comment} 

\usepackage{ dsfont }

\usepackage{color}

\usepackage[OT2,OT1]{fontenc}
\newcommand\cyr{%
\renewcommand\rmdefault{wncyr}%
\renewcommand\sfdefault{wncyss}%
\renewcommand\encodingdefault{OT2}%
\normalfont
\selectfont}
\DeclareTextFontCommand{\textcyr}{\cyr}

\usepackage{amssymb,amsmath}

\DeclareFontFamily{OT1}{rsfs}{}
\DeclareFontShape{OT1}{rsfs}{n}{it}{<-> rsfs10}{}
\DeclareMathAlphabet{\mathscr}{OT1}{rsfs}{n}{it}

\numberwithin{equation}{section}
\newtheorem{theorem}{Theorem}[section]

\newtheorem{cor}[theorem]{Corollary}

\newtheorem{conjecturbe}[theorem]{Conjecture}

\newtheorem{prop}[theorem]{Proposition}

\theoremstyle{definition}
\newtheorem{defn}[theorem]{Definition}
\theoremstyle{remark}
\newtheorem{remark}[theorem]{Remark}

\newtheorem{example}[theorem]{Example}

\newcommand{\Ass}{\operatorname{Ass}}
\newcommand{\Add}{\operatorname{Add}}

\newcommand{\im}{\operatorname{im}}

\newcommand{\tr}{\operatorname{tr}}

\newcommand{\pd}{\operatorname{pd}}
\newcommand{\id}{\operatorname{id}}
\newcommand{\Ext}{\operatorname{Ext}}
\newcommand{\Supp}{\operatorname{Supp}}

\newcommand{\Hom}{\operatorname{Hom}}

\newcommand{\End}{\operatorname{End}}

\newcommand{\p}{\mathfrak{p}}

\newcommand{\m}{\mathfrak{m}}

\newcommand{\w}{\omega}

\begin{document}
\title[Generalized Trace Submodules and Centers of Endomorphism Rings]{Generalized Trace Submodules and Centers of Endomorphism Rings}

\author[Lyle]{Justin Lyle}
\email[Justin Lyle]{jlyle106@gmail.com}
\urladdr{https://jlyle42.github.io/justinlyle/}



\begin{abstract}

Let $R$ be a commutative Noetherian local ring and $M$ a finitely generated $R$-module. We introduce a general form of the classically studied trace map that unifies several notions from the literature. We develop a theory around these objects and use it to provide a broad extension of a result of Lindo calculating the center of $\End_R(M)$. As a consequence, we show under mild hypotheses that in dimension $1$, the canonical module of $Z(\End_R(M))$ may be calculated as the trace submodule of $M$ with respect to the canonical module of $R$.
    
\end{abstract}

\maketitle

\section{Introduction}




Let $(R,\m,k)$ be a Noetherian local ring and let $M$ be a finitely generated $R$-module. The map $\Hom_R(M,R) \otimes_R M \to R$ given by $f \otimes x \mapsto f(x)$ is known as the \emph{trace map} of $M$ whose image $\tr_R(M)$ is the corresponding \emph{trace ideal} of $M$. These classically studied objects have received a renewed attention in recent years, and have been applied to the study of several famous and longstanding open conjectures. For instance, \cite{Li172} uses the theory of trace ideals to show the Auslander-Reiten conjecture holds for ideals in an Artinian Gorenstein ring, while \cite{Li17} uses them to show the Huneke-Wiegand conjecture holds for an ideal of $I$ of positive grade if the commutative ring $\End_R(I)$ is Gorenstein. A key point in the theory of trace ideals is a result of Lindo that shows $Z(\End_R(M)) \cong \End_R(\tr_R(M))$ when $M$ is faithful and reflexive \cite[Theorem 3.9]{Li17}.

In this paper, we introduce a broad generalization of the trace map and study it via the natural left module structure of $\End_R(M)$ on $M$. Concretely, we develop a theory around the natural map \[\phi^M_{L,N}:\Hom_R(M,N) \otimes_{\End_R(M)} \Hom_R(L,M) \to \Hom_R(L,N)\] of $\End_R(N)-\End_R(L)$ bimodules and its corresponding image which we denote by $\tr_{L,N}(M)$ (see Definition \ref{tracedef}). The map $\phi^M_{L,N}$ and its image $\tr_{L,N}(M)$ generalize and unify several disparate and well-studied notions from the literature (see Example \ref{unify}). Our main theorem regarding these objects provides a vast extension of \cite[Theorem 3.9]{Li17} mentioned previously. The following is a key special case of our main theorem (Theorem \ref{general}):
\begin{theorem}\label{introthm1}
Suppose $M$ is $N$-reflexive, i.e., the natural map $M \to \Hom_{\End_R(N)}(\Hom_R(M,N),N)$ is an isomorphism, and suppose for all $\p \in \Ass_R(R)$, that one of $R_{\p}$ or $N_{\p}$ is a direct summand of $M_{\p}^{\oplus n_{\p}}$ for some $n_{\p}$. Then $Z(\End_R(M)) \cong \End_{\End_R(N)}(\tr_{R,N}(M))$ as $R$-algebras. 
\end{theorem}



As a consequence of Theorem \ref{introthm1}, we show that if $R$ is Cohen-Macaulay with canonical module $\w$, and $M$ is maximal Cohen-Macaulay, then under a mild local condition, e.g., $M$ has rank, we have $Z(\End_R(M)) \cong \End_R(\tr_{R,\w}(M))$. If moreover $R$ has dimension $1$, then we also obtain that $\tr_{R,\w}(M)$ is a canonical module for $Z(\End_R(M))$ (see Corollary \ref{canonical}). Through our work, we offer some insights on the famous Huneke-Wiegand conjecture (see Conjectures \ref{hwc} and \ref{hwcalt} and Proposition \ref{equiv}).

\section{Preliminaries}

Throughout, we let $(R,\m,k)$ be a commutative Noetherian local ring, and we let $M$, $N$, and $L$ be finitely generated $R$-modules. We note that $M$ carries a natural left module action of $\End_R(M)$, given by $f \cdot x=f(x)$. Further, $\Hom_R(M,N)$ carries the structure of an $\End_R(N)-\End_R(M)$ bimodule by function composition in the natural way. We recall the $R$-module $M$ is said to have \emph{rank} $r$ if $M_{\p} \cong R_{\p}^{\oplus r}$ for all $\p \in \Ass_R(R)$. We let $\mu_R(M)$ denote the minimal number of generators of $M$ over $R$. We write $(-)^{*}:=\Hom_R(-,R)$ and, if $R$ is Cohen-Macaulay with a canonical module $\w$, we write $(-)^{\vee}:=\Hom_R(-,\w)$. If $E$ is a possibly noncommutative ring, we let $Z(E)$ denote the center of $E$. We note that $Z(\End_R(M))=\End_{\End_R(M)}(M)$. 

The following two definitions provide key notions in our study of generalized trace maps:
\begin{defn}
\
\begin{enumerate}
\item We say $M$ \emph{covariantly generates} $N$ with respect to $L$ if there is an $n$ and a map $p:M^{\oplus n} \to N$ such that $\Hom(L,p)$ is surjective. 
\item We say $M$ \emph{contravariantly generates} $N$ with respect to $L$ if there is an $n$ and a map $q:N \to M^{\oplus n}$ so that $\Hom_R(q,L)$ is surjective. 
\end{enumerate}We simply say that $M$ generates $N$ if $M$ generates $N$ covariantly with respect to $R$, and this is equivalent to the existence of a surjection $p:M^{\oplus n} \to N$ for some $n$. We say $M$ is a \emph{generator with respect to $L$} if $M$ covariantly generates every $R$-module $N$ with respect to $L$, and we refer to $M$ as a \emph{generator} if it is a generator with respect to $R$. When $R$ is local, we note $M$ is a generator if and only if $R$ is a direct summand of $M$.
\end{defn}


\begin{defn}

We let 
\[\epsilon_{M,N}^L:\Hom_R(M,N) \to \Hom_{\End_R(L)}(\Hom_R(L,M),\Hom_R(L,N))\] be the natural map of $\End_R(N)-\End_R(M)$ bimodules given by $\epsilon_{M,N}^L(f)=\Hom_R(L,f)$. We let 
\[\pi_{M,N}^L:\Hom_R(M,N) \to \Hom_{\End_R(L)}(\Hom_R(N,L),\Hom_R(M,L))\] be the natural map of $\End_R(N)-\End_R(M)$ bimodules given by $\pi_{M,N}^L(f)=\Hom_R(f,L)$. We define the following conditions for the ordered pair of $R$-modules $(M,N)$ with respect to the $R$-module $L$:
\begin{enumerate}
\item $(M,N)$ is said to be \emph{covariantly $L$-torsionless} if $\epsilon^L_{M,N}$ is injective. 

\item $(M,N)$ is said to be \emph{contravariantly $L$-torsionless} if $\pi^L_{M,N}$ is injective.

\item $(M,N)$ is said to be \emph{covariantly $L$-reflexive} if $\epsilon^L_{M,N}$ is an isomorphism.

\item $(M,N)$ is said to be \emph{contravariantly $L$-reflexive} if $\pi^L_{M,N}$ is an isomorphism.

\end{enumerate}

When $M=R$, the map $\pi^L_{R,N}$ may be identified, through the natural isomorphisms $\Hom_R(R,N) \cong N$ and $\Hom_R(R,L) \cong L$, with the natural evaluation map $\pi^L_N$. That is to say, 
\[\pi^L_N:N \to \Hom_{\End_R(L)}(\Hom_R(N,L),L)\] is the map of $\End_R(N)-R$ bimodules given by $\pi^L_N(x)=\alpha_x$, where $\alpha_x:\Hom_R(N,L) \to L$ is given by $\alpha_x(g)=g(x)$. We simply say that $N$ is $L$-torsionless (resp. $L$-reflexive) if $\pi^L_N$ is injective (resp. is an isomorphism). We note that $N$ is $R$-torsionless (resp. $R$-reflexive) if and only if it is torsionless (resp. reflexive) is the traditional sense. 

\end{defn}

We let $\operatorname{mod}(R)$ be the category of finitely generated $R$-modules, and we let $\Add_R(M)$ be the full subcategory of $\operatorname{mod}(R)$ consisting of finite direct sums of direct summands of $M$. So $N$ is an object of $\Add_R(M)$ if and only if $N$ is a summand of $M^{\oplus n}$ for some $n$.

\begin{remark}
We remark that the maps $\epsilon^L_{M,N}$ and $\pi^L_{M,N}$ are in a sense more natural than the canonical maps $\iota^L_{M,N}:\Hom_R(M,N) \to \Hom_R(\Hom_R(L,M),\Hom_R(L,N)$ and $\kappa^L_{M,N}:\Hom_R(M,N) \to \Hom_R(\Hom_R(N,L),\Hom_R(M,L))$ studied by several in the literature (see e.g. \cite[Definition 3.11]{LT23}). Indeed, the maps $\iota^L_{M,N}$ and $\kappa^L_{M,N}$ factor through $\epsilon^L_{M,N}$ and $\pi^L_{M,N}$, respectively, to begin with, and images of the latter are meaningfully closer to $\Hom_R(M,N)$ in most circumstances. For example, if $R$ is a domain and the rank of $L$ is positive, then the $R$-modules $\Hom_R(M,N)$, $\Hom_{\End_R(L)}(\Hom_R(L,M),\Hom_R(L,N)$, and $\Hom_{\End_R(L)}(\Hom_R(N,L),\Hom_R(M,L))$ all have the same rank by Morita considerations, but the ranks of the modules $\Hom_R(\Hom_R(L,M),\Hom_R(L,N))$ and $\Hom_R(\Hom_R(N,L),\Hom_R(M,L))$ will agree with that of $\Hom_R(M,N)$ if and only if the rank of $L$ is $1$. In particular, one cannot hope for $\iota^L_{M,N}$ or $\kappa^L_{M,N}$ to be an isomorphism in such a situation. 
\end{remark}

The following definition gives the main object of consideration for this work.

\begin{defn}\label{tracedef}

We let $\phi_{L,N}^M:\Hom_R(M,N) \otimes_{\End_R(M)} \Hom_R(L,M) \to \Hom_R(L,N)$ be the natural map of $\End_R(N)-\End_R(L)$ bimodules given on elementary tensors by $\phi_{L,N}^M(f \otimes g)=f \circ g$, and we refer to it as the \emph{trace map of $M$ with respect to the pair $(L,N)$}. We let $\tr_{L,N}(M)$ denote the image of $\phi^M_{L,N}$ and refer to it as the \emph{trace submodule of $\Hom_R(L,N)$ associated to $M$}.

\end{defn}

We give context to these objects through some remarks and examples.

\begin{remark}\label{trdesc}

We note that $\tr_{L,N}(M)$ can be described as 
\[\tr_{L,N}(M)=\{f \in \Hom_R(L,N) \mid f \mbox{ factors through } M^{\oplus n} \mbox{ for some } n\}.\]

\end{remark}

\begin{remark}

The map $\phi^M_{L,N}$ is the special case of the familiar Yoneda $\Ext$-map 
\[\Ext^p_R(M,N) \otimes_{\End_R(M)} \Ext^q_R(L,M) \to \Ext^{p+q}_{R}(L,N)\] (see e.g. \cite[Section 2]{Oo64} for the definition) with $p$ and $q$ taken to be $0$. The behavior of $\phi^M_{L,N}$ is much more predictable when compared with Yoneda maps for higher values of $p$ and $q$, due to the nature of higher extensions. 


\end{remark}

The map $\phi^M_{L,N}$ unifies several objects of consideration from the literature. The following are some natural contexts in which $\phi^M_{L,N}$ and $\tr_{L,N}(M)$ have been studied:

\begin{example}\label{unify}
\
\begin{enumerate}

\item[$(1)$] When $L=R$, the map $\phi^M_{R,N}$ can be naturally identified, through the canonical isomorphisms $\Hom_R(R,L) \cong L$ and $\Hom_R(R,N) \cong N$, with the map $\phi^M_N:\Hom_R(M,N) \otimes_{\End_R(M)} M \to N$ given on elementary tensors by $\phi^M_N(f \otimes x)=f(x)$. We let $\tr_N(M)$ denote the image of $\phi^M_N$, which is an $\End_R(N)-R$ subbimodule of $N$.

\item[$(2)$] When $L=N=R$, the map $\phi^M_R:M^{*} \otimes_{\End_R(M)} M \to R$ is the usual trace map from the literature whose image $\tr_R(M)$ is the corresponding \emph{trace ideal} of $M$; see e.g. \cite{Li17,Fa20}.

\item[$(3)$] When $M=R$, the map $\phi^M_{L,N}$ may be identified, similar to Part $(1)$, with the map $\theta_{L,N}:N \otimes_R L^{*} \to \Hom_R(L,N)$ given on elementary tensors by $x \otimes f \mapsto \alpha_{x,f}$, where $\alpha_{x,f}:N \to N$ is given by $\alpha_{x,f}(y)=f(y)x$. This map has been studied extensively in the context of stable module theory as developed by Auslander-Bridger \cite{AB69}. 

\item[$(4)$] If $R$ is Cohen-Macaulay with canonical module $\w$, the case where $L=R$ and $N=\w$ has been studied in the context of generalizations of the well-known Huneke-Wiegand conjecture (Conjecture \ref{hwc}). See \cite[Conjecture 1.3]{GT15}, \cite[Remark 2.5]{GT17}, and \cite[Question 4.1]{CG19} (cf. the discussion after Conjecture \ref{hwc} below).

\item[$(5)$] If $C$ is a \emph{semidualizing module}, i.e., $\Ext^i_R(C,C)=0$ for all $i>0$ and the natural homothety map $R \to \End_R(C)$ is an isomorphism, then the condition that $\phi_M^C$ be an isomorphism is part of the requirements for $M$ to belong to the Bass class of $C$; see \cite[Definition 3.1.4]{KS09} or \cite[1.8]{TW10}.

\end{enumerate}

\end{example}

The following is well-known, but we provide a short proof due to lack of a suitable reference.
\begin{prop}\label{faithref}
If $M$ is torsionless, then $M$ is faithful if and only if $R_{\p} \in \Add_{R_{\p}}(M_{\p})$ for all $\p \in \Ass_R(R)$.     
\end{prop}

\begin{proof}
It's clear that if $R_{\p} \in \Add_{R_{\p}}(M_{\p})$ for all $\p \in \Ass_R(R)$ that $M$ must be faithful. For the converse, suppose $M$ is faithful so that there is an injection $j:R \to M^{\oplus r}$ for some $r$. As $M$ is torsionless, there is an injection $M \xrightarrow{i} R^{\oplus n}$ for some $n$. Then there is an exact sequence of the form $0 \rightarrow R \xrightarrow{i^{\oplus r} \circ j} R^{\oplus nr} \to C \to 0$, so $\pd_R(C) \le 1$. In particular, $C_{\p}$ has finite projective dimension for all $\p \in \Ass(R)$, but as $R_{\p}$ has depth zero, the Auslander-Buchsbaum formula forces $C_{\p}$ to be a free $R_{\p}$-module. In particular $i^{\oplus r} \circ j$ is a split injection locally at every $\p \in \Ass(R)$. Then so is $j$, and so $R_{\p} \in \Add_{R_{\p}}(M_{\p})$ for all $\p \in \Ass(R)$, as desired.
\end{proof}
\section{Main Results}
We begin this section by recording some of the key properties of the map $\phi^M_{L,N}$ and its corresponding trace submodules, most of which will be needed for the proof of our main theorem.

\begin{prop}\label{tracetheory}
For any $R$-modules $A,B,L,M,N$, we have
\begin{enumerate}
    \item[$(1)$] $\tr_{L,N}(A \oplus B)=\tr_{L,N}(A)+\tr_{L,N}(B)$.
    \item[$(2)$] If $L$ or $N$ is in $\Add_R(M)$, then $\phi^M_{L,N}$ is an isomorphism. In particular, if $M$ is a generator then $\phi^M_N$ is an isomorphism.
    \item[$(3)$] If $A$ generates $B$ covariantly with respect to $L$ or contravariantly with respect to $N$, then $\tr_{L,N}(B) \subseteq \tr_{L,N}(A)$.
    \item[$(4)$] 
    If $M$ generates $N$ covariantly with respect to $L$ or if $M$ generates $L$ contravariantly with respect to $N$ then $\phi_{L,N}^M$ is surjective.
    \item[$(5)$] Let $i:\tr_{L,N}(M) \to \Hom_R(L,N)$ denote the natural inclusion. If the pair $(M,N)$ is covariantly reflexive with respect to $L$ (resp. the pair $(L,M)$ is contravariantly reflexive with respect to $N$), then $\Hom_{\End_R(L)}(\tr_{L,N}(M),i)$ (resp. $\Hom_{\End_R(N)}(\tr_{L,N}(M),i)$) is an isomorphism. 
    
    \item[$(6)$] If $S$ is a flat $R$-algebra, then $\phi^M_{L,N} \otimes_R S$ may be identified with $\phi^{M \otimes_R S}_{L \otimes_R S,N \otimes_R S}$, and we have $\tr^M_{L,N} \otimes_R S=\tr^{M \otimes_R S}_{L \otimes_R S,N \otimes_R S}$ upon identifying $\Hom_R(L,N) \otimes_R S$ with $\Hom_S(L \otimes_R S,N \otimes_R S)$. In particular, $\phi^M_{L,N}$ and $\tr^M_{L,N}$ respect localization and completion over $R$.

\end{enumerate}
    
\end{prop}

\begin{proof}

Note $(1)$ is immediate given Remark \ref{trdesc}. For $(2)$, if $L \in \Add_R(M)$, then there is a split injection $i:L \to M^{\oplus n}$ for some $n$, with retraction $p:M^{\oplus n} \to L$. There is a commutative diagram:
\[\begin{tikzcd}
	{\Hom_R(M,N) \otimes_{\End_R(M)} \Hom_R(M^{\oplus n},M)} && {\Hom_R(M^{\oplus n},N)} \\
	{\Hom_R(M,N) \otimes_{\End_R(M)} \End_R(M)^{\oplus n}} && {\Hom_R(M,N)^{\oplus n}}
	\arrow["{\phi^M_{M^{\oplus n},N}}", from=1-1, to=1-3]
	\arrow["\eta", from=2-1, to=2-3]
	\arrow["{a}"', from=1-1, to=2-1]
	\arrow["b"', from=1-3, to=2-3]
\end{tikzcd}\]
where $a,b$ and $\eta$ are the natural isomorphisms. In particular, $\phi^M_{M^{\oplus n},N}$ is an isomorphism. But there is also a commutative diagram
\[\begin{tikzcd}
	{\Hom_R(M,N) \otimes_{\End_R(M)} \Hom_R(L,M)} && {\Hom_R(L,N)} \\
	{\Hom_R(M,N) \otimes_{\End_R(M)} \Hom_R(M^{\oplus n},M)} && {\Hom_R(M^{\oplus n},N)} \\
	{\Hom_R(M,N) \otimes_{\End_R(M)} \Hom_R(L,M)} && {\Hom_R(L,N)}
	\arrow["{\phi^M_{L,N}}", from=1-1, to=1-3]
	\arrow["{\phi^M_{M^{\oplus n},N}}", from=2-1, to=2-3]
	\arrow["{\Hom_R(M,N) \otimes_{\End_R(M)} \Hom_R(i,M)}"', from=1-1, to=2-1]
	\arrow["{\Hom_R(i,N)}", from=1-3, to=2-3]
	\arrow["{\phi^M_{L,N}}", from=3-1, to=3-3]
	\arrow["{\Hom_R(p,N)}", from=2-3, to=3-3]
	\arrow["{\Hom_R(M,N) \otimes_{\End_R(M)} \Hom_R(p,M)}"', from=2-1, to=3-1]
\end{tikzcd}\]
As $i$ is a split injection the commutativity of the top square forces $\phi^M_{L,N}$ to be injective, while that of the bottom square forces it to be surjective, so it is an isomorphism. The case where $N \in \Add_R(M)$ instead follows from a nearly identical argument.

For $(3)$, suppose $A$ generates $B$ covariantly with respect to $L$, so there is a map $p:A^{\oplus m} \to B$ such that $\Hom_R(L,p)$ is a surjection. Suppose we have $g \in \tr_{L,N}(B)$, so that there is a commutative diagram  
\[\begin{tikzcd}
	L && N \\
	& {B^{\oplus n}}
	\arrow["t"', from=1-1, to=2-2]
	\arrow["s"', from=2-2, to=1-3]
	\arrow["g", from=1-1, to=1-3]
\end{tikzcd}\]
for some $n$. Since $\Hom_R(L,p)$ is surjective, so is $\Hom_R(L,p^{\oplus n})$, and thus there is a map $q \in \Hom_R(L,A^{\oplus mn})$ so that $p^{\oplus n} \circ q=t$. Then there is a commutative diagram
\[\begin{tikzcd}
	L && N \\
	& {A^{\oplus mn}}
	\arrow["q"', from=1-1, to=2-2]
	\arrow["{s \circ p^{\oplus n}}"', from=2-2, to=1-3]
	\arrow["g", from=1-1, to=1-3]
\end{tikzcd}\]
which shows $g \in \tr_{L,N}(A)$. The case where $A$ generates $B$ contravariantly with respect to $N$ follows similarly.

For $(4)$, from part (3) we have either $\tr_{L,N}(L) \subseteq \tr_{L,N}(M)$ when $M$ generates $N$ covariantly with respect to $L$ or $\tr_{L,N}(N) \subseteq \tr_{L,N}(M)$ when $M$ generates $L$ contravariantly with respect to $N$. But from part $(2)$, we have $\tr_{L,N}(L)=\tr_{L,N}(N)=\Hom_R(L,N)$, so $\phi^M_{L,N}$ is surjective.

For $(5)$, suppose the pair $(M,N)$ is covariantly reflexive with respect to $L$. Let $t_1,\dots,t_n$ be a minimal generating set for $\Hom_R(M,N)$ as a right $\End_R(M)$-module and let $p:(\End_R(M))^{\oplus n} \to \Hom_R(M,N)$ be the surjection mapping each standard basis element $e_i$ to $t_i$. For each $i=1,\dots,n$, let $j_i:\Hom_R(L,M) \to (\End_R(M))^{\oplus n} \otimes_{\End_R(M)} \Hom_R(L,M)$ be given by $j_i(a)=e_i \otimes a$. Now pick $f \in \Hom_{\End_R(L)}(\tr_{L,N}(M),\Hom_R(L,N))$. We claim $f$ factors through the map $i$, equivalently, that $\im f \subseteq \tr_{L,N}(M)$. For this, it suffices to show the images under $f$ of the generators $f_i \circ a$ for $a \in \Hom_R(L,M)$ are contained in $\tr_{L,N}(M)$. For each $i=1,\dots,n$, set $w_i=f \circ \phi^M_{L,N} \circ (p \otimes \id_{\Hom_R(L,M)}) \circ j_i$. Then each $w_i \in \Hom_{\End_R(L)}(\Hom_R(L,M),\Hom_R(L,N))$, and as $(M,N)$ is covariantly reflexive with respect to $L$ there are unique maps $s_i \in \Hom_R(M,N)$ so that $w_i=\Hom_R(L,s_i)$. Now given a generator $f_i \circ a$ of $\tr_{L,N}(M)$, we have $f_i \circ a=\phi^M_{L,N}((p \otimes \id_{\Hom_R(L,M)})(j_i(a)))$, so $f(f_i \circ a)=w_i(a)=\Hom_R(L,s_i)(a)=s_i \circ a \in \tr_{L,N}(M)$, and the claim follows. The case where $(L,M)$ is contravariantly reflexive with respect to $N$ instead follows from a similar argument.

The claims of $(6)$ follow from the commutative diagram
\[\begin{tikzcd}[column sep=3em]
	{(\Hom_R(M,N) \otimes_{\End_R(M)} \Hom_R(L,M)) \otimes_R S} & {\Hom_R(L,N) \otimes_R S} \\
	{\Hom_S(M \otimes_R S,N \otimes_R S) \otimes_{\End_S(M \otimes_R S)} \Hom_S(L \otimes_R S,M \otimes_R S)} & {\Hom_S(L \otimes_R S,N \otimes_R S)}
	\arrow["{\phi^M_{L,N} \otimes_R S}", from=1-1, to=1-2]
	\arrow[from=1-1, to=2-1]
	\arrow[from=1-2, to=2-2]
	\arrow["{\phi^{M \otimes_R S}_{L \otimes_R S,N \otimes_R S}}", from=2-1, to=2-2]
\end{tikzcd}\]

whose vertical arrows are the natural isomorphisms owed to the flatness of $S$.

\end{proof}

\begin{remark}
Proposition \ref{tracetheory} recovers several well-known facts about trace ideals, and clarifies that many of the familiar properties of trace ideals are owed to the freeness of $L$ and/or $N$ in considering $\tr_{L,N}(M)$.
    
\end{remark}

We now present the main theorem of this section, from which we will derive several Corollaries, including Theorem \ref{introthm1}. The statement of this theorem carries some technicality, but allows for a great deal of flexibility in hypotheses. 

\begin{theorem}\label{general}

Suppose, for all $\p \in \Supp_R(N) \cap \Ass_R(L)$, that one of $L_{\p}$ or $N_{\p}$ is in $\Add_{R_{\p}}(M_{\p})$. Then:
\begin{enumerate}
\item[$(1)$] If the pair $(M,N)$ is covariantly reflexive with respect to $L$, then there is an isomorphism of rings $\End_{\End_R(L)}(\tr_{L,N}(M)) \cong \End_{\End_R(M)}(\Hom_R(M,N))$.
\item[$(2)$] If the pair $(L,M)$ is contravariantly reflexive with respect to $N$, then there is an isomorphism of rings $\End_{\End_R(N)}(\tr_{L,N}(M)) \cong \End_{\End_R(M)}(\Hom_R(L,M))$

\end{enumerate}

\end{theorem}

\begin{proof}

There is a short exact sequence of $\End_R(N)-\End_R(L)$ bimodules
\[\epsilon: 0 \rightarrow K \rightarrow \Hom_R(M,N) \otimes_{\End_R(M)} \Hom_R(L,M) \xrightarrow{\phi^M_{L,N}} \tr_{L,N}(M) \rightarrow 0.\] By Proposition \ref{tracetheory} (2) and (6), the hypotheses force $\phi^M_{L,N}$ to be an isomorphism locally at every $\p \in \Supp_R(L) \cap \Ass_R(N)=\Ass_R(\Hom_R(L,N))$. In particular, $K_{\p}=0$ for all $\p \in \Ass_R(\Hom_R(L,N))$, and it follows that $\Hom_R(K,\Hom_R(L,N))=0$. As the $R$-modules $\Hom_{\End_R(N)}(K,\Hom_R(L,N))$ and $\Hom_{\End_R(L)}(K,\Hom_R(M,N))$ are submodules of $\Hom_R(K,\Hom_R(L,N))=0$, it follows they are $0$ as well. 

For $(1)$, we apply $\Hom_{\End_R(L)}(-,\Hom_R(L,N))$ to $\epsilon$ to see that $\Hom_{\End_R(L)}(\phi^M_{L,N},\Hom_R(L,N))$ gives an isomorphism
$\Hom_{\End_R(L)}(\tr_{L,N}(M),\Hom_R(L,N)) \to \Hom_{\End_R(L)}(\Hom_R(M,N) \otimes_{\End_R(M)} \Hom_R(L,M),\Hom_R(L,N))$. By Hom-tensor adjointness, the right hand term is naturally isomorphic to
\[\Hom_{\End_R(M)}(\Hom_R(M,N),\Hom_{\End_R(L)}(\Hom_R(L,M),\Hom_R(L,N)).\]
But as the pair $(M,N)$ is covariantly reflexive with respect to $L$, this term is naturally isomorphic to $\End_{\End_R(M)}(\Hom_R(M,N))$. The claim then follows from Proposition \ref{tracetheory} (5), noting the composition $\End_{\End_R(L)}(\tr_{L,N}(M)) \to \End_{\End_R(M)}(\Hom_R(M,N))$ is a map of rings.

For (2), we follow a similar approach; applying $\Hom_{\End_R(N)}(-,\Hom_R(L,N))$ we observe that $\Hom_{\End_R(N)}(\phi^M_{L,N},\Hom_R(L,N))$ gives an isomorphism $\Hom_{\End_R(N)}(\tr_{L,N}(M),\Hom_R(L,N)) \to \Hom_{\End_R(N)}(\Hom_R(M,N) \otimes_{\End_R(M)} \Hom_R(L,M),\Hom_R(L,N))$
and Hom-tensor adjointness shows the right hand term is isomorphic to 
\[\Hom_{\End_R(M)}(\Hom_R(L,M),\Hom_{\End_R(N)}(\Hom_R(M,N),\Hom_R(L,N)).\]
Since the pair $(L,M)$ is contravariantly reflexive with respect to, this term is naturally isomorphic to $\End_{\End_R(M)}(\Hom_R(L,M))$, and we have the claim.

\end{proof}

\begin{cor}\label{lessgeneral}

Suppose, for all $\p \in \Supp_R(N) \cap \Ass_R(R)$, that one of $R_{\p}$ or $N_{\p}$ is in $\Add_{R_{\p}}(M_{\p})$. If $M$ is reflexive with respect to $N$, then there is a ring isomorphism $\End_{\End_R(N)}(\tr_{N}(M)) \cong Z(\End_R(M))$.

\end{cor}

\begin{proof}

The claim follows immediately from Theorem \ref{general} (2) with $L=R$.

\end{proof}

Applying Corollary \ref{lessgeneral} we immediately recover the result of Lindo mentioned in the introduction.
\begin{theorem}[{\cite[Theorem 3.9] 
{Li17}}]\label{evenlessgeneral}

Suppose $M$ is faithful and reflexive. Then there is a ring isomorphism $\End_R(\tr_R(M)) \cong Z(\End_R(M))$.

\end{theorem}

\begin{proof}


It follows from Proposition \ref{faithref} that $R_{\p} \in \Add_{R_{\p}}(M_{\p})$ for all $\p \in \Ass(R)$. The claim thus follows from Corollary \ref{lessgeneral} with $N=R$.

\end{proof}

A significant special case of Theorem \ref{general} is the following which also serves as a variation on Theorem \ref{evenlessgeneral}.

\begin{cor}\label{canonical}

Suppose $R$ is Cohen-Macaulay with canonical module $\w$ and set $(-)^{\vee}=\Hom_R(-,\w)$. Suppose $M$ is reflexive with respect to $\w$, e.g. $M$ is maximal Cohen-Macaulay, and suppose for every $\p \in \Ass_R(R)$ that one of $M_{\p}$ or $M^{\vee}_{\p}$ is a generator. Then there is an isomorphism of rings $Z(\End_R(M)) \cong \End_R(\tr_{\w}(M))$. If moreover $R$ has dimension $1$, then we have $\tr_{\w}(M)$ is a canonical module for $Z(\End_R(M))$.

\end{cor}

\begin{proof}
For the first claim, it suffices from Theorem \ref{general} to show, for every $\p \in \Ass_R(R)$, that one of $R_{\p}$ or $\w_{\p}$ is in $\Add_{R_{\p}}(M_{\p})$. But this follows immediately from noting that $R_{\p} \in \Add_{R_{\p}}(M_{\p})$ when $M_{\p}$ is a generator and that $w_{\p} \in \Add_{R_{\p}}(M_{\p})$ when $M^{\vee}_{\p}$ is a generator, since we may dualize a splitting $M^{\vee}_{\p} \to R_{\p}$. So we have $\End_R(\tr_{\w}(M)) \cong Z(\End_R(M))$. If $R$ has dimension  $1$, then we note $Z(\End_R(M))$ embeds in a direct sum of copies of $M$, so is maximal Cohen-Macualay. Note that $\End_R(\tr_{\w}(M)) \cong \Hom_R(\tr_{\w}(M),\w)$ from Proposition \ref{tracetheory} (5), since any pair $(M,N)$ is covariantly reflexive with respect to $R$. So from \cite[Theorem 3.3.7 (b)]{BH93}, we have that a canonical module for $Z(\End_R(M))$ is $\Hom_R(Z(\End_R(M)),\w) \cong \Hom_R(\End_R(\tr_{\w}(M),\w) \cong \Hom_R(\Hom_R(\tr_{\w}(M),w),\w) \cong \tr_{\w}(M)$, with the last isomorphism owed to the fact that $\tr_{\w}(M)$ embeds in $\w$, and is thus maximal Cohen-Macaulay when $R$ has dimension $1$.    
\end{proof}

\begin{remark}

The condition that one of $M_{\p}$ or $M^{\vee}_{\p}$ be a generator for every $\p \in \Ass_R(R)$ is required since a module that is reflexive with respect to $\w$ need not be torsionless without additional hypotheses, unlike reflexivity with respect to $R$. This condition is quite mild and can be guaranteed by assuming, for instance, that $M$ is torsionless locally at every associated prime of $R$. This holds, for example, if $R$ is generically Gorenstein. 
    
\end{remark}

Our work provides some perspective on the famous Huneke-Wiegand conjecture. We recall the Huneke-Wiegand conjecture may be stated in one of its more general forms as follows:

\begin{conjecturbe}\label{hwc}
Suppose $R$ is a Cohen-Macaulay local ring of dimension $1$. If $M$ is a torsion-free $R$-module with rank such that $M^{*} \otimes_R M$ is torsion-free, then $M$ is free.
\end{conjecturbe}

The assumption that $M$ be torsion-free is not really needed, as one may reduce to this case, but it is convenient technically. Conjecture \ref{hwc} is generally considered more mysterious, and less is known, if we do not assume $R$ to be Gorenstein. Indeed, if $R$ is Gorenstein then $R$ is its own canonical module and so its duality properties are better behaved. It has thus not been clear whether the formulation of Conjecture \ref{hwc} is suitable for this level of generality, and there have been some attempts at providing alternate variations. For instance, a version was considered by\cite{GT15} for ideals where the hypothesis that $\Hom_R(I,R) \otimes_R I$ is torsion-free is replaced by the hypothesis that $\Hom_R(I,\w) \otimes_R I$ is torsion-free, and the conclusion that $I$ is free is replaced by the conclusion that either $I$ is free or $I \cong \w$. A counterexample was provided in \cite[Example 7.3]{GT15}, however, that shows this conjecture need not hold in general, even for a reasonably well-behaved numerical semigroup ring. Our purpose for the remainder of this section is to provide some indication that the statement of Conjecture \ref{hwc} is natural even outside the Gorenstein setting. To this end, we pose the following:

\begin{conjecturbe}\label{hwcalt}
Suppose $R$ is a Cohen-Macaulay local ring of dimension $1$. If $M$ is a torsion-free $R$-module with rank such that $M^{\vee} \otimes_R \Hom_R(\w,M)$ is torsion-free, then $M \cong \w^{\oplus n}$ for some $n$.
    
\end{conjecturbe}

The naturality of this conjecture is evidenced by the existence of the map $M^{\vee} \otimes_R \Hom_R(\w,M) \to \Hom_R(\w,\w) \cong R$ which factors through the map $\phi^M_{\w,\w}$ introduced above, and whose image $\tr^M_{\w,\w}$ may be identified with an ideal of $R$ that serves as a type of twisted trace ideal for $M$. In the case where $M$ is an ideal of $R$, the hypothesis that $M^{\vee} \otimes_R \Hom_R(\w,M)$ be torsion-free will force $M^{\vee} \otimes_{\End_R(M)} \Hom_R(\w,M)$ to be torsion-free which in turn will force $\phi^M_{\w,\w}$ to be injective, since it is so at every $\p \in \Ass_R(R)$ from Proposition \ref{tracetheory} $(2)$ and $(6)$. 
Our claim that Conjecture \ref{hwc} is natural comes then from the following observation:

\begin{prop}\label{equiv}
Conjectures \ref{hwc} and \ref{hwcalt} are equivalent.
\end{prop}

\begin{proof}

Since $M$ is torsion-free with $\dim R=1$, it is maximal Cohen-Macaulay, so $M$ is reflexive with respect to $\w$. Then, by Hom-tensor adjointness, we observe
\[M^{\vee} \otimes_R \Hom_R(\w,M) \cong M^{\vee} \otimes_R \Hom_R(\w,M^{\vee \vee}) \cong M^{\vee} \otimes_R \Hom_R(\w \otimes_R M^{\vee},\w) \cong M^{\vee} \otimes_R (M^{\vee})^{*}.\]
Thus, the requirement that $M^{\vee} \otimes_R \Hom_R(\w,M)$ be torsion-free is equivalent to the hypotheses that $(M^{\vee})^{*} \otimes_R M^{\vee}$ is torsion-free, and $M \cong \w^{\oplus n}$ for some $n$ if and only if $M^{\vee}$ is free. So the two conjectures are equivalent, as desired.  
\end{proof}

\bibliographystyle{amsalpha}
\bibliography{mybib}

\vspace{.5cm}

\end{document}